\documentclass%[11pt]
{article}
\usepackage{amsmath,amsfonts,amssymb,amsthm}
\usepackage{latexsym, xspace, enumerate}
\usepackage[mathscr]{eucal}
\usepackage[all]{xypic}
\usepackage{hyperref}
\usepackage{amscd}
\usepackage{vmargin}

\setmarginsrb{28mm}{20mm}{28mm}{25mm}%
          {10mm}{7mm}{10mm}{10mm}

\newtheorem{theorem}{Theorem}[section]

\newtheorem{claim}[theorem]{Claim}
\newtheorem{lemma}[theorem]{Lemma}
\newtheorem{fact}[theorem]{Fact}
\newtheorem{corollary}[theorem]{Corollary}

\theoremstyle{definition}

\newtheorem{example}[theorem]{Example}
\newtheorem{remark}[theorem]{Remark}%[section]

\newcommand{\Z}{\mathbb Z}

\def\depth{\mathrm{depth}}

\def\Im{\mathrm{Im}}

\newcommand{\NBD}{}

\numberwithin{equation}{section}

\title{Limit free computation of entropy}

\author{Dikran Dikranjan
\\{\footnotesize {\tt  dikran.dikranjan@uniud.it}} 
\\{\footnotesize Dipartimento di Matematica e Informatica,}
\\{\footnotesize Universit\`{a} di Udine,}
\\{\footnotesize Via delle Scienze, 206 - 33100 Udine, Italy} 
\and Anna Giordano Bruno
\\{\footnotesize {\tt  anna.giordanobruno@uniud.it}} 
\\{\footnotesize Dipartimento di Matematica e Informatica,}
\\{\footnotesize Universit\`{a} di Udine,}
\\{\footnotesize Via delle Scienze, 206 - 33100 Udine, Italy}
}
\date{Dedicated to the sixtieth birthday of Fabio Zanolin}

\begin{document}

\maketitle

\begin{abstract}
Various limit-free formulas are given for the computation of the algebraic and the topological entropy, respectively in the settings of endomorphisms of locally finite discrete groups and of continuous endomorphisms of totally disconnected compact groups. As applications we give new proofs of the connection between the algebraic and the topological entropy in the abelian case and of the connection of the topological entropy with the finite depth for topological automorphisms.
\end{abstract}

\bigskip
\noindent Key words: topological entropy, algebraic entropy, totally disconnected compact group, finite depth\\
2010 AMS Subject Classification: 37B40, 22C05, 54H11, 54H20, 54C70, 20K30.

\section{Introduction}

In this paper we are concerned with the topological and the algebraic entropy respectively in the setting of continuous endomorphisms of totally disconnected compact groups and of endomorphisms of locally finite groups. In the abelian case the correspondence between these two settings - that is between continuous endomorphisms of totally disconnected compact abelian groups and endomorphisms of torsion abelian groups - is given by Pontryagin duality.

\medskip
%Inspired by the notion of measure entropy in Ergodic Theory given by Kolmogorov \cite{K} and Sinai \cite{Sinai}, 
In \cite{AKM} Adler, Konheim and McAndrew introduced the topological entropy for continuous selfmaps of compact spaces, while later on Bowen in \cite{B} introduced it for uniformly continuous selfmaps of metric spaces, and this definition was extended to uniformly continuous selfmaps of uniform spaces by Hood in \cite{hood}. As explained in details in \cite{DG-islam}, for continuous endomorphisms of totally disconnected compact groups the topological entropy can be introduced as follows. It is worth recalling that a totally disconnected compact group $K$ has as a local base at $1$ the family $\mathcal B(K)$ of all open subgroups of $K$, as proved by van Dantzig in \cite{vD} for all locally compact totally disconnected groups.

Let $K$ be a totally disconnected compact group and $\psi:K\to K$ a continuous endomorphism. For every open subgroup $U$ of $K$ and every positive integer $n$ let $$C_n(\psi,U)=U\cap \psi^{-1}(U)\cap\ldots\cap\psi^{-n+1}(U)$$ be the \emph{$n$-th $\psi$-cotrajectory} of $U$, and the \emph{$\psi$-cotrajectory} of $U$ is
$$C(\psi,U)=\bigcap_{n=1}^\infty\psi^{-n}(U)=\bigcap_{n=1}^\infty C_n(\psi,U).$$
Note that this is the greatest $\psi$-invariant subgroup of $K$ contained in $U$.

The \emph{topological entropy of $\psi$ with respect to $U$} is given by the following limit, which is proved to exist (see also Lemma \ref{logalpha} below),
$$H_{top}(\psi,U)=\lim_{n\to \infty}\frac{\log[K:C_n(\psi,U)]}{n}.$$
The \emph{topological entropy} of $\psi$ is $$h_{top}(\psi)=\sup\{H_{top}(\psi,U):U\in\mathcal B(K)\}.$$

\medskip
Using ideas briefly sketched in \cite{AKM}, Weiss developed in \cite{W} the definition of algebraic entropy for endomorphisms of torsion abelian groups. Moreover Peters modified this definition in \cite{Pet} for automorphisms of abelian groups, and this approach was extended to all endomorphisms of abelian groups in \cite{DG2}; in \cite{DG-islam} also the hypothesis of commutativity of the groups was removed. Following \cite{DG-islam} we give here the definition of algebraic entropy for endomorphisms of locally finite groups, which coincides with the definition given in \cite{AKM} in the abelian case. 

Let $G$ be locally finite group and $\phi:G\to G$ an endomorphism. 
Denote by $\mathcal F(G)$ the family of all finite subgroups of $G$. For every finite subgroup $F$ of $G$ and every positive integer $n$ let
$$
T_n(\phi,F)=F\cdot \phi(F)\cdot\ldots\cdot\phi^{n-1}(F)
$$ 
be the \emph{$n$-th $\phi$-trajectory} of $F$, and the \emph{$\phi$-trajectory} of $F$ is
$$
T(\phi,F)=
%\sum_{n=1}^\infty T_n(\phi,F)=
\bigcup_{n=1}^{\infty}T_n(\phi,F).
$$
If $G$ is abelian, then $T(\phi,F)$ is the smallest $\phi$-invariant subgroup of $G$ containing $F$.

The \emph{algebraic entropy of $\phi$ with respect to $F$} is the following limit, which exists as proved in \cite{DG-islam},
$$
H_{alg}(\phi,F)={\lim_{n\to\infty}\frac{\log|T_n(\phi,F)|}{n}}.
$$
The \emph{algebraic entropy} of $\phi$ is 
$$
h_{alg}(\phi)=\sup\{H_{alg}(\phi,F):F\in\mathcal F(G)\}.
$$

Every locally finite group is obviously torsion, while the converse holds true under the hypothesis that the group is abelian; in the other hand the solution of Burnside's problem shows that even groups of finite exponent may fail to be locally finite.

\medskip
Yuzvinski claims at the end of his paper \cite{Y},  that for every torsion abelian group $G$ and every endomorphism $\phi:G\to G$ one has
\begin{equation}\label{dag'}
h_{alg}(\phi)=\sup \left\{\log \left |\frac{T(\phi,F)}{\phi(T(\phi,F))} \right| : F\in\mathcal F(G)\right\}.
\end{equation}
This formula is {\em false} without the assumption that $\phi$ is injective, as shown by Example \ref{YN} below (see also \cite{DSV}).
The huge gap in Example \ref{YN} is due to the special choice of the zero endomorphism. In fact, as noted in \cite{DSV}, Yuzvinski's claim is true for {\em injective} endomorphisms. 
A proof of this theorem, based on a much more general result on multiplicities, was given in \cite{DSV}. Here we offer a short multiplicity-free proof of the following more general and precise formula that obviously implies the theorem. Note that if $G$ is a torsion abelian group and $\phi:G\to G$ an endomorphism, then the hypothesis that $\ker\phi\cap T(\phi,F)$ in the following formula is automatically satisfied (see Lemma \ref{kerfinite}).

\bigskip
\noindent\textbf{Algebraic Formula.}
\emph{Let $G$ be a  locally finite group, $\phi:G\to G$ an endomorphism and $F$ a normal finite subgroup of $G$ such that $\ker\phi\cap T(\phi,F)$ is finite. Then 
$$
H_{alg}(\phi,F) = \log \left |\frac{T(\phi,F)}{\phi(T(\phi,F))} \right|-\log|\ker\phi\cap T(\phi,F)|.
$$}
\smallskip

The next corollary shows that Yuzvinski's claim holds true for injective endomorphisms. 

\begin{corollary}\label{Coro0:May21}
Let $G$ be a locally finite group, $\phi:G\to G$ an injective endomorphism and $F$ a finite normal subgroup of $G$.  Then
%If $\phi\restriction_{T(\phi,F)}:T(\phi,F)\to T(\phi,F)$ is injective, then
$$
H_{alg}(\phi,F) = \log \left |\frac{T(\phi,F)}{\phi(T(\phi,F))} \right|.
$$
Therefore \eqref{dag'} holds true whenever $\phi$ is injective. 
\end{corollary}

This formula suggests a similar approach for the topological entropy. Indeed it is possible to prove the following limit-free formula for the topological entropy. Also in this case, if the totally disconnected compact group $K$ is abelian and $\psi:K\to K$ is a continuous endomorphism, then the condition that ${K}/({\Im\psi+C(\psi,U)})$ is finite is automatically satisfied.

\bigskip
\noindent\textbf{Topological Formula.}
\emph{Let $K$ be a totally disconnected compact group, $\psi:K\to K$ a continuous endomorphism and $U$ an open normal subgroup of $K$ such that ${K}/({\Im\psi\cdot C(\psi,U)})$ is finite. Then $$H_{top}(\psi,U)=\log\left|\frac{\psi^{-1}(C(\psi,U))}{C(\psi,U)}\right|-\log\left|\frac{K}{\Im\psi\cdot C(\psi,U)}\right|.$$}
\smallskip

Stoyanov in \cite{St} proved that in the compact case for the computation of the topological entropy one can reduce to surjective endomorphisms $\psi$, for which the quotient ${K}/({\Im\psi\cdot C(\psi,U)})$ is obviously trivial.
The much simpler formula in this case (practically, the topological counterpart of Corollary \ref{Coro0:May21}) is given in Corollary \ref{Coro1:May21}.

%The next theorem allows us to reduce to \emph{surjective} continuous selfmaps of compact spaces in the computation of the topological entropy.
%For a compact space $X$ and a continuous selfmap $\psi:X\to X$, let $$E_\psi(X)=\bigcap_{n\in\N}\psi^n(X);$$
%
%\begin{theorem}[Reduction to surjective selfmaps] \label{red-to-sur}
%Let $X$ be a compact space and $\psi:X\to X$ be a continuous selfmap.  Then $E_\psi(X)$ is closed and $\psi$-invariant, the map $\psi\restriction_{E_\psi(X)}:E_\psi(X)\to E_\psi(X)$ is surjective and $h_{top}(\psi)=h_{top}(\psi\restriction_{E_\psi(X)})$.
%\end{theorem}

\medskip
In Section \ref{alg-sec} we give a proof of the Algebraic Formula, while in Section \ref{top-sec} we verify the Topological Formula.  Moreover we note how these two results give immediately a new proof of Weiss Bridge Theorem connecting the algebraic and the topological entropy by Pontryagin duality. Note that the Pontryagin dual of a torsion abelian group is a totally disconnected compact abelian group.

\bigskip
\noindent\textbf{Weiss Bridge Theorem.}
\emph{Let $G$ be a torsion abelian group and $\phi:G\to G$ an endomorphism. Let $K=\widehat G$ be the Pontryagin dual of $G$ and let $\psi=\widehat\phi:K\to K$ be the dual of $\phi$. Then $$h_{alg}(\phi)=h_{top}(\psi).$$}
\medskip

%\medskip
%The Topological Formula applied to topological automorphisms gives immediately the following
%
%\bigskip
%\noindent\textbf{Corollary.}
%\emph{Let $K$ be a totally disconnected compact group, $\psi:K\to K$ a topological automorphism and $U$ an open normal subgroup of $K$. Then $$H_{top}(\psi,U)=\log\left|\frac{C(\psi,U)}{\psi(C(\psi,U))}\right|.$$}
%\bigskip

%\subsection{Finite depth automorphisms of profinite groups}

Let $K$ be a totally disconnected compact group and $\psi:K\to K$ a topological automorphism. In \cite{Willis4} Willis defined the pair $(K,\psi)$ to have \emph{finite depth} if there exists $U\in\mathcal B(K)$ such that 
\begin{equation}\label{Eq5:May21}
\bigcap_{n\in\Z}\psi^n(U)=\{1\};
\end{equation}
 we call a subgroup $U$ with this property \emph{$\phi$-antistable}. One can show that $K$ must necessarily be metrizable and  totally disconnected (see Section \ref{depth-sec} for more details). For a pair $(K,\psi)$ of finite depth the \emph{depth} of $\psi$ is 
\begin{equation}\label{Eq4:May21}
\depth(\psi)=[\psi(C(\psi^{-1},U)):C(\psi^{-1},U)];
\end{equation}
 as noted in \cite{Willis4} this index is finite and does not depend on the choice of the $\phi$-antistable subgroup $U\in\mathcal B(K)$.

In Section \ref{depth-sec} an application of the Topological Formula is given in Theorem \ref{h=logdp}, stating that in case $(K,\psi)$ is a pair of finite depth, then
 $$
 h_{top}(\psi)=\log\depth(\psi).
$$ 
%In \cite{Willis4} for this result is given a reference to \cite{Ki}\NB, where a proof is given for the measure theoretic entropy. 
A similar result for the measure-theoretic entropy, going into a somewhat different direction, can be found in  \cite[Theorem 2]{Ki}.
\NBD According to Halmos \cite{Halmos}, surjective continuous endomorphisms of compact groups are measure preserving, and in this case the measure theoretic entropy coincides with the topological entropy as proved by Stoyanov \cite{St}.
%According to Fact \ref{Halmos}, both $h_{top}$ and $h_{mes}$ are available for surjective continuous endomorphisms of compact groups.
%It was proved that they coincide by Berg \cite{Berg} for surjective endomorphisms of tori, by Aoki \cite{Aoki} for automorphisms of compact abelian groups. The proof of the general case was obtained by Stoyanov \cite{S} as an application of the Uniqueness Theorem \ref{UT-top}.

%\medskip
%We use the multiplicative notation for arbitrary groups and the additive notation for abelian groups.

\section{Algebraic entropy}\label{alg-sec}

The following example shows that Yuzvinski's claim \eqref{dag'} is false without the assumption that the considered endomorphism is injective.

\begin{example}\label{YN}
Let $G$ be a non-zero torsion abelian group, and $\phi:G\to G$ be the zero endomorphism. Take any non-zero finite subgroup $F$ of $G$;
then $T(\phi,F)=F$ and $\phi(T(\phi,F))=0$. 
%The right-hand-side of \eqref{dag'} is greater than or equal to $\log|T(\phi,F)/\phi(T(\phi,F))|=\log|F|,$ while $h_{alg}(\phi) = 0$. 
Then $$\sup \left\{\log \left |\frac{T(\phi,F)}{\phi(T(\phi,F))} \right| : F\in\mathcal F(G)\right\}\geq \log\left|\frac{T(\phi,F)}{\phi(T(\phi,F))}\right|=\log|F|,$$ while $h_{alg}(\phi) = 0$.

In particular, when $G$ is an {\em infinite} torsion abelian group, for every $n>0$ we can pick a finite subgroup $F_n$ of $G$ of size $\geq n$. 
%Then the right-hand-side of \eqref{dag'} is $\geq \sup_n \log n = \infty$. So one has $0 = h_{alg} (\phi) \ne \infty$ in \eqref{dag'}.
Then $$\sup \left\{\log \left |\frac{T(\phi,F)}{\phi(T(\phi,F))} \right| : F\in\mathcal F(G)\right\}=\infty,$$ so in this case one has $0 = h_{alg} (\phi) \ne \infty$ in \eqref{dag'}.
\end{example}

We are now in position to prove the Algebraic Formula.

\begin{theorem}[Algebraic Formula]\label{AYF2}
Let $G$ be a locally finite group, $\phi:G\to G$ an endomorphism and $F$ a finite normal subgroup of $G$ such that $\ker\phi\cap T(\phi,F)$ is finite.  Then
%If $\phi\restriction_{T(\phi,F)}:T(\phi,F)\to T(\phi,F)$ is injective, then
$$H_{alg}(\phi,F) = \log \left |\frac{T(\phi,F)}{\phi(T(\phi,F))} \right|-\log|\ker\phi\cap T(\phi,F)|.$$
\end{theorem}

\begin{proof} Let $K=\ker\phi\cap T(\phi,F)$, which is finite by hypothesis. We show that one can assume without loss of generality that $F$ contains $K$. Indeed let $F'=F K\subseteq T(\phi,F)$. Then $T(\phi,F')=T(\phi,F)$, and so $H_{alg}(\phi,F')=H_{alg}(\phi,F)$; moreover $K=\ker\phi\cap F'\subseteq F'$. So we assume without loss of generality that $K\subseteq F$ and we verify that 
$$
H_{alg}(\phi,F) = \log \left|\frac{T(\phi,F)}{\phi(T(\phi,F))} \right|-\log|K|.
$$ 

For the sake of brevity, we write in the sequel $T_n$ and $T$, for $T_n(\phi,F)$ and $T(\phi,F)$ respectively. 

Arguing as in \cite[Lemma 1.1]{DGSZ} we have that the index $[T_{n+1}:T_n]$ stabilizes, i.e.,  there exists $n_0>0$ such that for all $n>n_0$ one has $[T_{n+1}:T_n]=\alpha$, consequently $H_{alg}(\phi,F) = \log \alpha$.  Our aim is to show that also 
\begin{equation}\label{(1)}
\left |\frac{T}{\phi(T)} \right|= \alpha\cdot|K|;
\end{equation}
obviously this proves the theorem.  Since $T = F\cdot \phi(T)$ and $(F\cdot \phi (T))/\phi (T)\cong F/(F \cap \phi(T))$, it follows that \eqref{(1)} is equivalent to 
\begin{equation}\label{(2)}
\left|\frac{F}{F \cap \phi (T)}\right|= \alpha\cdot |K|.
\end{equation} 
The increasing chain $F\cap \phi (T_n)$  of finite subgroups of $F$ stabilizes, so there exists $n_1>0$ such that $F\cap \phi (T) = F \cap \phi (T_n)$
for all $n\geq n_1$.  Hence \eqref{(2)} is equivalent to $$\left|\frac{F}{F \cap \phi(T_n)}\right|= \alpha\cdot |K|$$ for all $n\geq n_1$. 

As $ F/(F \cap \phi (T_n)) \cong ( F\cdot \phi (T_n))/\phi (T_n) = T_{n+1}/\phi (T_n)$, we conclude that 
\begin{equation}\label{(3)}
\left|\frac{F}{F \cap \phi (T_n)}\right|= [T_{n+1}: \phi (T_n)].
\end{equation}
%Since $\phi\restriction_{T}$ is injective, 
Since $\phi(T_n)\cong T_n/(\ker\phi\cap T_n)=T_n/K$, we have 
$|\phi (T_n)|\cdot |K|= |T_n|$. Hence Lagrange Theorem applied to the group $T_{n+1}$ and its subgroups $T_n$ and $\phi (T_n)$ gives
\begin{equation}\label{(4)}
[T_{n+1}: \phi (T_n)] = \frac{|T_{n+1}|}{ |\phi (T_n)|}= \frac{|T_{n+1}|\cdot |K|}{ |T_n|}= [T_{n+1}: T_n] \cdot |K|=\alpha\cdot |K|,
\end{equation} 
provided $n\geq \max \{n_0,n_1\}$. From \eqref{(3)} and \eqref{(4)} we get \eqref{(2)}, and this concludes the proof.
\end{proof}

\section{Topological entropy}\label{top-sec}

The following is the counterpart of \cite[Lemma 1.1]{DGSZ} for the topological entropy. Its proof follows the one of \cite[Lemma 2.2]{DGS}.

\begin{lemma}\label{logalpha}
Let $K$ be a compact group, $\psi:K\to K$ a continuous endomorphism and $U$ an open normal subgroup of $K$. For every positive integer $n$ let $c_n:=[K:C_n(\psi,U)]$. Then 
\begin{itemize}
\item[(a)] $c_{n}$ divides $c_{n+1}$ for every $n>0$. 
\end{itemize}
For every $n>0$ let $\alpha_{n}:=\frac{c_{n+1}}{c_{n}}=[C_{n}(\psi,U):C_{n+1}(\psi,U)]$. Then
\begin{itemize}
\item[(b)] $\alpha_{n+1}$ divides $\alpha_{n}$ for every $n>0$. 
\item[(c)] Consequently the sequence $\{\alpha_n\}_{n>0}$ stabilizes, i.e., there exist integers $n_0>0$ and $\alpha>0 0$ such that $\alpha_n=\alpha$ for every $n\geq n_0$. 
\item[(d)] Moreover $H_{top}(\psi,U)=\log\alpha$.
\item[(e)] If $\psi$ is a topological automorphism, $H_{top}(\psi^{-1},U)= H_{top}(\psi,U)$. 
\end{itemize}
\end{lemma}
\begin{proof}
Let $n>0$. Since there is no possibility of confusion we denote $C_n(\psi,U)$ simply by $C_n$.

\smallskip
(a) Since $K/C_n$ is isomorphic to $(K/C_{n+1})/(C_n/C_{n+1})$, it follows that $\frac{c_{n+1}}{c_n}=[C_{n}:C_{n+1}]$ and in particular $c_{n}$ divides $c_{n+1}$.

\smallskip
(b) We prove that $C_n/C_{n+1}$ is isomorphic to a subgroup of $C_{n-1}/C_n$, and this gives immediately the thesis. First note that 
$$\frac{C_n}{C_{n+1}}=\frac{C_n}{C_n\cap \psi^{-n}(U)}\cong \frac{C_n\cdot\psi^{-n}(U)}{\psi^{-n}(U)}.$$ Now $(C_n\psi^{-n}(U))/\psi^{-n}(U)$ is a subgroup of the quotient $(\psi^{-1}(C_{n-1})\psi^{-n}(U))/\psi^{-n}(U)$. The homomorphism $K/\psi^{-n}(U)\to K/\psi^{-n+1}(U)$ induced by $\psi$ is injective, therefore 
$(\psi^{-1}(C_{n-1})\cdot\psi^{-n}(U))/\psi^{-n}(U)$ is isomorphic to its image $$\frac{C_{n-1}\cdot\psi^{-n+1}(U)}{\psi^{-n+1}(U)}\cong \frac{C_{n-1}}{C_{n-1}\cap \psi^{-n+1}(U)}=\frac{C_{n-1}}{C_n}.$$

\smallskip
(c) follows immediately from (b).

\smallskip
(d) By item (c) for $n_0>0$ we have $c_{n_0+n}=\alpha^n c_{n_0}$ for every $n\geq0$, and by the definition of topological entropy 
$$H_{top}(\psi,U)=\lim_{n\to\infty}\frac{\log c_n}{n}=\lim_{n\to \infty}\frac{\log(\alpha^n c_{n_0})}{n}=\log\alpha.$$

\smallskip
(e) Assume that $\psi$ is a topological automorphism. For every positive integer $n$ let $c_n^*:=[K:C_n(\psi^{-1},U)]$.
According to (a)--(c) applied to $\psi^{-1}$, $H_{top}(\psi^{-1},U)=\log \alpha^*$, where $\alpha^*$ is the value at which stabilizes the sequence 
$\alpha^*_n:=\frac{c_{n+1}^*}{c_{n}^*}$. Hence it suffices to see that $c_n^* = c_n$ for all $n>0$ and this is obvious
since $\psi^{n-1}(C_n(\psi,U))= C_n(\psi^{-1},U)$. 
%c_n:=[K:C_n(\psi,U)] = [K:U][U:]$ and $c_n^*:=[K:C_n(\psi^{-1},U)] = [K:U][U:C_n(\psi^{-1},U)]$. 
%it suffices to see that $H_{top}(\psi^{-1},U)= H_{top}(\psi,U)$. 
\end{proof}

For the  proof of the Topological Formula we need the following folklore fact that we give with a proof for reader's convenience.

\begin{lemma}\label{intersect}
Let $G$ be a topological group and let $T$ be a closed subset of $G$. Then for every descending chain 
$B_1 \supseteq B_2 \supseteq \ldots \supseteq B_n \supseteq  \ldots$ of closed subsets of $G$ the intersection $B = \bigcap_{n=1}^\infty B_n $ is non-empty and 
 $\bigcap_{n=1}^\infty (B_{n} T)= B T$, whenever $B_1$ is countably compact. 
\end{lemma}
\begin{proof} 
That $B\ne \emptyset$ is a direct consequence of the countable compactness of $B_1$. 

The inclusion $\bigcap_{n=1}^\infty (B_{n} T)\supseteq B T$ is obvious. To verify the converse inclusion pick an element $x\in \bigcap_{n=1}^\infty (B_{n} T)$. Then there exist elements
$b_n\in B_n$, $t_n \in T$ such that $x = b_nt_n$  for every $n>0$. Let $D_n$ be the closure of the set $\{b_n, b_{n+1}, \ldots \}$ for $n>0$. Then 
\begin{equation}\label{Eq_19Maggio}
D_n \subseteq B_n\  \  \mbox{ for each }\  \ n>0. 
\end{equation}
The countable compactness of $B_1$ yields that $\bigcap_{n=1}^\infty D_n\ne \emptyset$. Fix an element $b$ of this intersection and note that 
$b \in B$ due to (\ref{Eq_19Maggio}). It suffices to prove that $b^{-1}x \in T$. Since $T$ is closed it suffices to check that $b^{-1}x$ belongs to the closure of $T$. 
To this end let $V= V^{-1}$ be a symmetric neighborhood of the neutral element of $G$. Then $bV$ is a neighborhood of $b\in D_1$, so $bV\ni b_m$ for some $m >0$. 
This yields $Vb^{-1}  \ni b_m^{-1}$, and consequently $Vb^{-1}x  \ni b_m^{-1}x = t_m$. Therefore $Vb^{-1} x\cap T \ne \emptyset$, and so $b^{-1}x$ belongs to the closure of $T$. 
\end{proof}

%We can now prove the Topological Formula. The following claim is needed in a crucial point of the proof.
%
%\begin{claim}\label{intersect}
%Let $K$ be a compact group and $U$ an open normal subgroup of $G$. Let $\{C_n\}_{n>0}$ be a decreasing chain of compact normal subgroups of $K$.
%Then $\bigcap_{n>0}(C_nU)=(\bigcap_{n>0}C_n)U$. 
%\end{claim}
%\begin{proof}
%The inclusion $\bigcap_{n>0}(C_nU)\supseteq (\bigcap_{n>0}C_n)U$ is obvious. So let $x\in \bigcap_{n>0}(C_nU)$. For every $n>0$ there exist $c_n\in C_n$ and $u_n\in U$ such that $x=c_nu_n$. Since $K$ is compact the sequence $\{c_n\}_{n>0}$ has an accumulation point $c$ and $c\in \bigcap_{n>0}C_n$ because $\{C_n\}_{n>0}$ is a decreasing chain. In particular $c^{-1}x$ is an accumulation point of the sequence $\{c_n^{-1}x\}_{n>0}$ and $c_n^{-1}x=u_n\in U$ for every $n>0$. Hence $c^{-1}x\in U$, and so $x\in cU\subseteq (\bigcap_{n>0}C_n)U$.
%\end{proof}

\begin{theorem}[Topological Formula]\label{TYF2}
Let $K$ be a totally disconnected compact group, $\psi:K\to K$ a continuous endomorphism and $U$ an open normal subgroup of $K$ such that $K/(\Im\psi\cdot C(\psi,U))$ is finite. Then $$H_{top}(\psi,U)=\log\left|\frac{\psi^{-1}(C(\psi,U))}{C(\psi,U)}\right|-\log\left|\frac{K}{\Im\psi\cdot C(\psi,U)}\right|.$$
\end{theorem}

\begin{proof} Since there is no possibility of confusion we denote $C_n(\psi,U)$ and $C(\psi,U)$ simply by $C_n$ and $C$ respectively.
Let $L=\Im\psi \cdot C$.
We can assume without loss of generality that $U\subseteq L$. Indeed otherwise one can take $U'=U\cap L$. Then $U'$ is open since $L$ is open, being a closed subgroup of $K$ of finite index by hypothesis; moreover $C(\psi,U)=C(\psi,U')$ as $\psi^{-1}(U)=\psi^{-1}(U')$ and so $H_{top}(\psi,U)=H_{top}(\psi,U')$.

Let us note that our assumption $U\subseteq L$ and the inclusion $C \subseteq U$ imply 
\begin{equation}\label{F1:20Maggio}
L= \Im\psi \cdot  C\subseteq \Im\psi \cdot  C_n \subseteq  \Im \psi  \cdot  U \subseteq   \Im \psi  \cdot  C \cdot  U \subseteq L \cdot U = L.
\end{equation}
The homomorphism $K/\psi^{-1}(C_n)\to K/C_n$ induced by $\psi$ is injective and the image of $K/\psi^{-1}(C_n)$ is $\Im\psi \cdot C_n /C_n$.  As $\Im\psi \cdot C_n = L$ by (\ref{F1:20Maggio}), we get 
\begin{equation}\label{F1:21Maggio}
[K:\psi^{-1}(C_n)]=[L:C_n]. 
\end{equation}

By Lemma \ref{logalpha} there exist integers $n_0>0$ and $\alpha> 0$ such that $\alpha_n=\alpha$ for every $n\geq n_0$ and $H_{top}(\psi,U)=\log\alpha$. So it suffices to prove that 
\begin{equation}\label{aim}
\left|\frac{\psi^{-1}(C)}{C}\right|=\alpha\cdot |L|.
\end{equation}
We start noting that 
\begin{equation}\label{eqa}
\frac{\psi^{-1}(C)}{C}=\frac{\psi^{-1}(C)}{\psi^{-1}(C)\cap U}\cong \frac{\psi^{-1}(C)\cdot U}{U}.
\end{equation}
The quotient $K/U$ is finite and $\{(\psi^{-1}(C_n)\cdot U)/U\}_{n>0}$ is a descending chain of subgroups of $K/U$, hence it stabilizes, that is there exists $n_1>0$ such that $$\psi^{-1}(C_n)\cdot U=\psi^{-1}(C_{n_1})\cdot U\ \text{for every}\ n\geq n_1;$$ in other words $\bigcap_{n=1}^\infty (\psi^{-1}(C_{n})\cdot U)=\psi^{-1}(C_{n_1})\cdot U$. Lemma \ref{intersect} gives
 $$
 \bigcap_{n=1}^\infty (\psi^{-1}(C_{n})\cdot U)=\left(\bigcap_{n=1}^\infty \psi^{-1}(C_n)\right)\cdot U,
 $$
and $\bigcap_{n=1}^\infty \psi^{-1}(C_n)=\psi^{-1}(C)$,  therefore 
\begin{equation}\label{eqb}
\psi^{-1}(C)\cdot U=\psi^{-1}(C_n) \cdot U\ \text{for every}\ n\geq n_1.
\end{equation}
Let $n\geq\max\{n_0,n_1\}$. Then \eqref{eqa} and \eqref{eqb} give 
$$
\frac{\psi^{-1}(C)}{C}\cong \frac{\psi^{-1}(C)\cdot U}{U}=\frac{\psi^{-1}(C_n)\cdot U}{U} \cong \frac{\psi^{-1}(C_n)}{\psi^{-1}(C_n)\cap U}=\frac{\psi^{-1}(C_n)}{C_{n+1}}.
$$
Consequently 
\begin{equation}\label{eqc}
\left|\frac{\psi^{-1}(C)}{C}\right|=\left|\frac{\psi^{-1}(C_n)}{C_{n+1}}\right|=\frac{[K:C_{n+1}]}{[K:\psi^{-1}(C_n)]}.
\end{equation}
As $[K:C_n]=[K:L][L:C_n]$, (\ref{F1:21Maggio}) gives 
\begin{equation}\label{eqd}
[K:\psi^{-1}(C_n)]=\frac{[K:C_n]}{[K:L]}.
\end{equation}
So using \eqref{eqd} in \eqref{eqc}, and recalling that $n\geq n_0$, we can conclude that 
\begin{equation}\label{eqe}
\left|\frac{\psi^{-1}(C)}{C}\right|=\frac{[K:C_{n+1}]}{[K:C_n]}[K:L]=[C_n:C_{n+1}][K:L]=\alpha[K:L].
\end{equation}
i.e., the wanted equality announced in \eqref{aim}. \end{proof}

As noted in the Introduction, if $K$ is a totally disconnected compact group, $\mathcal B(K)$ is a local base at $1$. In this case also the subfamily $\mathcal B_\triangleleft(K)$ of $\mathcal B(K)$ of all normal open subgroups of $K$ is a local base at $1$. Indeed we have the following property, where for $U\in\mathcal B(K)$, the \emph{heart} $U_K$ of $U$ in $K$ is the greatest normal subgroup of $K$ contained in $U$.

\begin{claim}\label{heart}
Let $K$ be a compact group. If $U\in\mathcal B(K)$, then $U_K\in\mathcal B_\triangleleft(K)$. 
\end{claim}

Since for any $U,V\in\mathcal B(K)$, if $U\subseteq V$, then $H_{top}(\psi,V)\leq H_{top}(\psi,U)$, by the definition of topological entropy we immediately derive that it suffices to take the supremum when $U$ ranges in a local base at $1$ of $K$:

\begin{claim}\label{basesuff}
Let $K$ be a totally disconnected compact group, $\psi:K\to K$ a continuous endomorphism and $\mathcal B\subseteq \mathcal B(K)$ a local base at $1$. Then $h_{top}(\psi)=\sup\{H_{top}(\psi,U):U\in\mathcal B\}$.
\end{claim}

In particular $h_{top}(\psi)=\sup\{H_{top}(\psi,U):U\in\mathcal B_\triangleleft(K)\}$, so we immediately get Corollary \ref{Coro1:May21} of the Topological Formula for continuous surjective endomorphism (in particular, for topological automorphisms).

Following Willis \cite{Willis4}, when $\psi$ is clear, we denote $C(\psi,U)$ also by the shorter and more suggestive $U_-$, and we leave $U_+$ denote the $\psi^{-1}$-cotrajectory $C(\psi^{-1},U)= \bigcap_{n=0}^\infty \psi^n(U)$. We start using this notation from \eqref{Eq3:May21}, where the first equality follows from 
Theorem \ref{TYF2}
and the second one follows from Lemma \ref{logalpha}(e) and the first equality. 

\begin{corollary}\label{Coro1:May21}
Let $K$ be a totally disconnected compact group and $\psi:K\to K$ a continuous surjective endomorphism. Then
\begin{equation}\label{Eq3:May21}
 H_{top}(\psi,U) = \log [\psi^{-1}(U_-): U_-] 
% \left|\frac{\psi^{-1}(C(\psi,U))}{C(\psi,U)}\right|
\mbox{ and }[\psi^{-1}(U_-): U_-] = [\psi(U_+):U_+]
%\footnote{$[\psi^{-1}(C(\psi,U)): C(\psi,U)] = [\psi(C(\psi^{-1},U)): C(\psi^{-1},U)]$}
\end{equation}
 for every $U\in\mathcal B_\triangleleft(K)$. In particular
$$
h_{top}(\psi)=\sup\left\{\log [\psi^{-1}(U_-): U_-] 
%\left|\frac{\psi^{-1}(C(\psi,U))}{C(\psi,U)}\right|
:U\in\mathcal B_\triangleleft(K)\right\}.
$$
%\begin{align*}
%h_{top}(\psi)&=\sup\left\{\log\left|\frac{\psi^{-1}(C(\psi,U))}{C(\psi,U)}\right|:U\in\mathcal B_\triangleleft(K)\right\}\\
%&=\sup\left\{\log\left|\frac{C(\psi,U)}{\psi(C(\psi,U))}\right|:U\in\mathcal B_\triangleleft(K)\right\}.
%\end{align*}
\end{corollary}

\begin{remark}
One can easily obtain from this corollary the well-known ``logarithmic law" 
$$
H_{top}(\psi^k,U) =k H_{top}(\psi,U)
$$
for a continuous surjective endomorphism $\psi$ and an integer $k$. Indeed it suffices to note that 
$$
[\psi^{-k}(U_-): U_-]=[\psi^{-1}(U_-): U_-]^k.
$$
\end{remark}

The next corollary is dedicated to the case of an open subgroup $U$ with trivial $\psi$-cotrajectory $U_-= C(\psi,U)$. 

\begin{corollary}\label{Coro2:May21}
Let $K$ be a totally disconnected compact group, $\psi:K\to K$ a continuous endomorphism and $U\in\mathcal B_\triangleleft(K)$ with trivial $\psi$-cotrajectory $U_-$. If $\mathrm{Coker}\,\psi =  K/\Im\psi$ is finite, then 
 $$
 H_{top}(\psi,U) = \log\left|\ker \psi\right| - \log |\mbox{{\rm Coker}  }\psi|. 
 $$
\end{corollary}

The aim of the next remark is to clarify the significance of the hypothesis $|K/(\Im\psi\cdot U_-)|<\infty$ in Theorem \ref{TYF2}. See also Remark \ref{coker} for an interesting consequence of Corollary \ref{Coro2:May21}.

\begin{remark}\label{Rem1:May21}
Let $\psi:K\to K$ a continuous endomorphism of a totally disconnected compact group $K$ and let $U$ an open normal subgroup of $K$. \NBD
\begin{itemize}
   \item[(a)] Let $K_U = K/U_-$, 
   %C(\psi,U)$,
    let $q_U: K\to K_U$ be the canonical homomorphism and let $\psi_U: K_U \to K_U$ be 
the induced endomorphism. Clearly, $U_-$
%C(\psi,U)$
is $\psi$-invariant (but need not be stabilized by $\psi$) and $q_U(U)$ is $\psi_U$-antistable (actually, $q(U)_-=C(\psi_U,q(U))$ is trivial). Moreover, $K/(\Im\psi\cdot U_-)$ is finite precisely when $\Im \psi_U$ has finite index in $K_U$. More precisely, 
$$
K/(\Im\psi\cdot U_-)\cong K_U/(\Im\psi_U\cdot q(U)_-)=K_U/ \Im\psi_U=\mathrm{Coker}\,\psi_U,
$$ 
as $q(U)_-$ is trivial. So 
$$
[K:(\Im\psi\cdot U_-)]= 
%[K_U: (\Im\psi_U\cdot C(\psi_U,q(U))]=  
[K_U: \Im\psi_U] = |\mathrm{Coker}\,\psi_U|. 
$$
By Corollary \ref{Coro2:May21} the triviality of $q(U)_-$  gives 
 %$\frac{\psi^{-1}_U(C(\psi_U,q(U)))}{C(\psi_U,q(U))}= {\psi^{-1}(C(\psi_U,q(U)))}=\ker \psi_U$. Therefore, 
 $$
 H_{top}(\psi_U,q_U(U))=\log |\ker \psi_U| - \log |\mathrm{Coker}\,\psi_U|.
 $$
\item[(b)] Now let $N = U_-\cap U_+$, $K_{(U)} = K/N_U$ and let $p_U: K\to K_{(U)}$ be the canonical homomorphism. 
Clearly, $N_U$ is stabilized by $\psi$, the induced endomorphism $\psi_{(U)}$ of $K_{(U)}$ is injective and $p_U(U)$ is $\psi_{(U)}$-antistable  (one can see as before that $K/(\Im\psi\cdot U_-)$ is finite precisely when $\mathrm{Coker}\,\psi_{(U)} =K_{(U)}/\Im \psi_{(U)}$ is finite, etc.). One can use the pairs $(K_{(U)}, \psi_{(U)})$ of finite depth to present the pair  $(K, \psi)$ as an inverse limit of the pairs $(K_{(U)}, \psi_{(U)})$ with $U\in \mathcal B_\triangleleft(K)$ of finite depth (see \cite[Proposition 5.3]{Willis4}). 
% More precisely, $K/\Im\psi\cdot N_U\cong  
%K_{(U)}/\Im\psi_{(U)}\cdot C(\psi_{(U)},p_U(U))=K_{(U)}/ \Im\psi_U$, as $C(\psi_U,q(U))$ is trivial. So 
%$$
%[K:(\Im\psi\cdot C(\psi,U))]= 
%%[K_U: (\Im\psi_U\cdot C(\psi_U,q(U))]=  
%[K_U: \Im\psi_U],
%$$
% By Corollary \ref{Coro2:May21}, the triviality of $C(\psi_U,q(U))$  gives 
% %$\frac{\psi^{-1}_U(C(\psi_U,q(U)))}{C(\psi_U,q(U))}= {\psi^{-1}(C(\psi_U,q(U)))}=\ker \psi_U$. Therefore, 
% $$
% H_{top}(\psi_U,q_U(U))=\log \ker \psi_U - \log [K_U: \Im\psi_U].
% $$
%$$\log\left|\frac{\psi^{-1}(C(\psi,U))}{C(\psi,U)}\right|= \log\left|{\psi^{-1}(C(\psi,U))}\right|= \log \ker \psi$.
\end{itemize}
\end{remark}

\section{The abelian case}

When the groups are abelian the finiteness conditions in the Algebraic Formula and in the Topological Formula automatically satisfied. Indeed we have the following result, which applies directly for the Algebraic Formula and together with Lemma \ref{T-C} for the Topological Formula.

%The next result shows that the finiteness condition in the Algebraic Formula is automatically satisfied when the group is abelian.

\begin{lemma}\label{kerfinite}
Let $G$ be a torsion abelian group, $\phi:G\to G$ an endomorphism and $F$ a finite subgroup of $G$.
Then $\ker\phi\cap T(\phi,F)$ is finite.
\end{lemma}
\begin{proof}
Since $T(\phi,F)$ is $\phi$-invariant, we can assume without loss of generality that $G=T(\phi,F)$. This is a finitely generated $\Z[X]$-module. Therefore $\ker\phi$ is a finitely generated $\Z[X]$-module as well. Since the action of $\phi$ on $\ker \phi$ sends $\ker\phi$ to $0$, we have that $\ker\phi$ is a finitely generated $\Z$-module. Hence $\ker\phi$ is finite.
\end{proof}

We recall now some definitions and results from Pontryagin duality.
For a topological abelian group $G$ the Pontryagin dual $\widehat G$ of $G$ is the group $\mathrm{Chom}(G,\mathbb T)$ of the continuous characters of $G$  endowed with the compact-open topology \cite{P}. The Pontryagin dual of a discrete Abelian group is always compact, Moreover we recall that a finite abelian group is isomorphic to its dual and the the dual of a torsion abelian group is a totally disconnected compact abelian group.
If $\phi:G\to G$ is an endomorphism, its Pontryagin adjoint $\widehat\phi:\widehat G\to \widehat G$ is defined by $\widehat\phi(\chi)=\chi\circ\phi$ for every $\chi\in\widehat G$.
For a subset $H$ of $G$, the \emph{annihilator} of $H$ in $\widehat G$ is $H^\perp=\{\chi\in\widehat G:\chi(H)=0\}$.

%We collect here some known facts concerning the Pontryagin duality that we will use in what follows, and which are proved in \cite{DPS}, \cite{HR} and \cite{O3}.
\begin{fact}\label{pontr}
Let $G$ be an abelian group.
\begin{itemize}
\item[(a)]If $\{H_n\}_{n>0}$ are subgroups of $G$, then $(\sum_{n=1}^\infty H_n)^\perp\cong \bigcap_{n=1}^\infty H_n^\perp$. 
%and $(\bigcap_{n=1}^\inftyH_i)^\perp\cong\sum_{i=1}^n H_i^\perp$.
\end{itemize}
If $H$ a subgroup of $G$ and $\phi:G\to G$ an endomorphism, then:
\begin{itemize}
\item[(b)] $\widehat H\cong \widehat G/H^\perp$;
\item[(c)] $(\phi^n(H))^\perp=(\widehat\phi)^{-n}(H^\perp)$ for every $n\geq0$;
\item[(d)] $\ker\phi^\perp=\Im\widehat\phi$.
\item[(e)] If $H\subseteq L$ are subgroups of $G$, then $H^\perp/L^\perp\cong\widehat{L/H}$.
\end{itemize}
\end{fact}

\begin{lemma}\label{T-C}
Let $G$ be a torsion abelian group, $\phi:G\to G$ an endomorphism and $F$ a finite subgroup of $G$. Let $K=\widehat G$, $\psi=\widehat\phi$ and $U=F^\perp$. Then $U\in\mathcal B(K)$ and
%Let $K$ be a totally disconnected compact abelian group, $\psi:K\to K$ a continuous endomorphism and $U$ an open subgroup of $K$. Let $\phi=\widehat \psi$ and $F=U^\perp$. Then 
\begin{itemize}
\item[(a)] $T_n(\phi,F)^\perp=C_n(\psi,U)$ for every $n>0$, and $T(\phi,F)^\perp=C(\psi,U)$;
\item[(b)] $\ker\phi\cap T(\phi,F)\cong{K}/({\Im\psi+C(\psi,U)})$;
\item[(c)] ${T(\phi,F)}/{\phi(T(\phi,F))}\cong {\psi^{-1}(C(\psi,U))}/{C(\psi,U)}$.
\end{itemize}
\end{lemma}
\begin{proof}
The conclusions follow from Fact \ref{pontr}.
\end{proof}

Applying this lemma, the Algebraic Formula and the Topological Formula we can now give a short proof of Weiss Bridge Theorem connecting the algebraic and the topological entropy.

\begin{corollary}[Weiss Bridge Theorem]
Let $G$ be a torsion abelian group and $\phi:G\to G$ an endomorphism, let $K=\widehat G$ and $\psi=\widehat \phi$. Then $$h_{alg}(\phi)=h_{top}(\psi).$$
\end{corollary}
\begin{proof}
Let $K=\widehat G$ and $\psi=\widehat\phi$. Let $U$ be an open subgroup of $K$. Then $F$ is a finite subgroup of $G$. By Theorem \ref{AYF2} and Theorem \ref{TYF2} and by Lemma \ref{T-C} we can conclude that $H_{alg}(\phi,F)=H_{top}(\psi,U)$, hence $h_{alg}(\phi)=h_{top}(\psi)$.
%we have $H_{top}(\psi,U)=\log\left|\frac{\psi^{-1}(C(\psi,U))}{C(\psi,U)}\right|-\log\left|\frac{K}{\Im\psi+C(\psi,U)}\right|$ and $H_{alg}(\phi,F) = \log \left |\frac{T(\phi,F)}{\phi(T(\phi,F))} \right|-\log|\ker\phi\cap T(\phi,F)|$.
\end{proof}

\begin{remark}
Applying Pontryagin duality in the abelian case one can also derive the Topological Formula from the Algebraic Formula.
Indeed, let $K$ be a totally disconnected compact abelian group and $\psi:K\to K$ a continuous endomorphism. Let $G=\widehat K$, $\phi=\widehat \psi$ and $F=U^\perp$. Then $F$ is a finite subgroup of $G$. By Lemma \ref{T-C} we have that $K/C_n(\psi,U)\cong T_n(\phi,F)$ and so $$H_{top}(\psi,U)=H_{alg}(\phi,F).$$ By Theorem \ref{AYF2} $H_{alg}(\phi,F) = \log |{T(\phi,F)}/{\phi(T(\phi,F))} |-\log|\ker\phi\cap T(\phi,F)|$ and again Lemma \ref{T-C} gives $$|T(\phi,F)/\phi(T(\phi,F))|=|\psi^{-1}(C(\psi,U))/C(\psi,U)|\ \text{and}\ |\ker\phi\cap T(\phi,F)|=|K/(\Im\psi + C(\psi,U))|.$$ Therefore $H_{top}(\psi,U)=\log|\psi^{-1}(C(\psi,U))/C(\psi,U)|-\log|K/(\Im\psi+ C(\psi,U))|$, that is the Topological Formula.
\end{remark}

\section{An application: finite depth and topological entropy}\label{depth-sec}

%\pagebreak 

Let $K$ be a totally disconnected compact group and $\psi:K\to K$ a topological automorphism. As recalled in the Introduction, the pair $(K,\psi)$ has \emph{finite depth} if there exists an $\phi$-antistable $U\in\mathcal B(K)$ (see \eqref{Eq4:May21}).
By Claim \ref{heart} we can assume without loss of generality that $U$ is also normal, that is $U\in\mathcal B_\triangleleft(K)$.
This definition implies that 
\begin{equation}\label{baseU}
\text{the family $\mathcal B_U=\{U_n:n>0\}$, where $U_n:=C_n(\psi,U)\cap C_n(\psi^{-1},U)$, is a local base at $1$.}
\end{equation}
In particular $K$ turns out to be necessarily metrizable and  totally disconnected. Morevoer
$K$ is isomorphic to a subgroup $G_1$ of $F^\Z$, where $F$ is a finite group; if $\sigma$ denotes the left Bernoulli shift of
$F^\Z$, then $G_1$ is stabilized by $\sigma$ and under the identification of $G$ with $G_1$ one has $\psi = \sigma\restriction _{G_1}$ (see also \cite[Proposition 2]{Ki}).

\begin{claim}\emph{\cite[Proposition 5.5]{Willis4}}
Let $(K,\psi)$ be a pair of finite depth. If $U,W\in\mathcal B_\triangleleft(K)$ are $\phi$-antistable, then $[\psi (U_+) : U_+] = [\psi (W_+) : W_+]$.
%$$[\psi(C(\psi^{-1},U)):C(\psi^{-1},U)]=[\psi(C(\psi^{-1},W)):C(\psi^{-1},W)].$$
\end{claim}

In view of this claim one defines the \emph{depth} of a pair $(K,\psi)$ of finite depth as 
$$
\depth(\psi)=
%[\psi(C(\psi^{-1},U)):C(\psi^{-1},U)] = 
[\psi (U_+) : U_+]
$$
 for any $\phi$-antistable $U\in\mathcal B_\triangleleft(K)$. Moreover, since 
\begin{equation}\label{U+->U-}
[\psi (U_+) : U_+]=[\psi^{-1}(U_-):U_-]
\end{equation}
according to (\ref{Eq3:May21}), one can extend this definition to 
$$
\mathrm{depth}(\psi) = [\psi (U_+) : U_+]=[\psi^{-1}(U_-):U_-],
$$
where $U\in\mathcal B_\triangleleft(K)$ is any  $\phi$-antistable $U\in\mathcal B_\triangleleft(K)$. 

\begin{theorem}\label{h=logdp}
Let $(K,\psi)$ be a pair of finite depth. Then $$h_{top}(\psi)=\log\depth(\psi).$$
\end{theorem}
\begin{proof} 
Let $U\in\mathcal B_\triangleleft(K)$ be $\phi$-antistable. By \eqref{baseU} the family $\mathcal B_U$ is a local base at $1$. Moreover for any $n>0$ we have $H_{top}(\psi,U_n)=\log[\psi^{-1}(C(\psi,U_n)):C(\psi,U_n)]$ by Theorem \ref{TYF2}, therefore \eqref{U+->U-} gives $$H_{top}(\psi,U_n)=\log\depth(\psi).$$ Hence $h_{top}(\psi)=\log\depth(\psi)$ by Claim \ref{basesuff}.
\end{proof}

The equality $h_{top}(\psi) = h_{top}(\psi^{-1})$ from Lemma  \ref{logalpha}(e)  is well known for the topological entropy of automorphisms of compact groups, we obtain as a by-product the following fact. %mentioned in \cite{Willis4}

\begin{corollary}\label{Coro4:May21}
Let $(K,\psi)$ be a pair of finite depth. Then $\depth(\psi)=\depth(\psi^{-1}).$
\end{corollary}

%The above equality is given in \cite[Proposition 5.5]{Willis4}, although its two-line proof is somewhat too concise to be comprehensive. 
%is based on the fact that a topological automorphism of a compact group preserves the Haar measure;
%\NBD (lo nascondo --- vedi Corolloari \ref{Coro4:May21} sotto " . . .  in particular this show that $\depth(\psi)=\depth(\psi^{-1})$.")

Theorem \ref{h=logdp} and Corollary \ref{Coro2:May21} have the following consequence. According to \cite[Proposition 5.5]{Willis4}, if $K$ is infinite, then $\mathrm{depth}(\psi)>1$.

\begin{remark}\label{coker}
Let $K$ be a totally disconnected compact group, $\psi:K\to K$ a continuous endomorphism and $U\in\mathcal B_\triangleleft(K)$ with trivial $\psi$-cotrajectory $C(\psi,U)$. The triviality of $C(\psi,U)$ implies that $U$ is $\psi$-antistable. This yields that the pair $(K,\psi)$ has finite depth, so if $K$ is infinite, we have $H_{top}(\psi,U) =\log \depth(\psi) > 0$ by Theorem \ref{h=logdp}. In particular Corollary \ref{Coro2:May21} gives the \NBD non-obvious inequality $\log\left|\ker \psi\right| - \log |\mathrm{Coker}\,\psi| > 0$, i.e., $\psi $ is necessarily non-injective and $|\ker \psi | >|\mathrm{Coker}\,\psi|$. 
\end{remark}

\section*{Acknowledgements}

It is a pleasure to thank George Willis for sending us his preprint \cite{Willis4} and for inspiring us to prove Theorem \ref{h=logdp}.
We thank also the members of our Seminar on Dynamical Systems at the University of Udine for the useful discussions on this topic.

\end{document}